\documentclass[12pt]{amsart}
\usepackage{amsfonts}
\usepackage{amssymb}
\usepackage[letterpaper, left=2.5cm, right=2.5cm, top=2.5cm,
bottom=2.5cm,dvips]{geometry}
\usepackage{verbatim}
\usepackage{graphicx}
\newtheorem{theorem}{Theorem}
\theoremstyle{plain}

\newtheorem{conjecture}{Conjecture}

\numberwithin{equation}{section}

\begin{document}
	\title[From the 1-2-3 Conjecture to the Riemann Hypothesis]{From the 1-2-3 Conjecture to the Riemann Hypothesis}

	\author{Jaros\l aw Grytczuk}
	\address{Faculty of Mathematics and Information Science, Warsaw University
		of Technology, 00-662 Warsaw, Poland}
	\email{j.grytczuk@mini.pw.edu.pl}
	
	\dedicatory{Dedicated to Xuding Zhu for his 60th birthday.}

	\thanks{Supported by the Polish National Science Center Grant 2015/17/B/ST1/02660.}

	\begin{abstract}
This survey presents some combinatorial problems with number-theoretic flavor. Our journey starts from a simple graph coloring question, but at some point gets close to a dangerous territory of the Riemann Hypothesis. We will mostly focus on open problems, but we will also provide some simple proofs, just for adorning.
	\end{abstract}
	
	\maketitle

\section{Introduction} The following question was asked by Karo\'{n}ski, \L uczak, and Thomason in \cite{KaronskiLT JCTB}: \emph{Can the edges of any non-trivial graph be assigned weights from $\{1,2,3\}$ so that adjacent vertices have different sums of incident edge weights?} In other words, the problem is to get a proper coloring of a graph using as colors vertex degrees in a multigraph obtained by locally bounded edge multiplexing.

The problem caught the eye of many researchers, but despite substantial effort it remains open. Many variations emerged in which one tries to get a graph coloring by other manipulations on vertex degrees. Unexpectedly, some of them appear to be related to deep number theoretic problems, like Graham's gcd-Problem \cite{Graham AMM}, the Erd\H{o}s Discrepancy Problem \cite{Erdos} (see \cite{Soundararajan}), or even the Riemann Hypothesis (see \cite{BorweinCRW}). We will present here some selected examples out of a rich variety of problems and results around that matter.

\section{Variations on the 1-2-3 Conjecture}
In this section we present a selection of problems related to the 1-2-3 Conjecture.
\subsection{Edge decorations}
Let $G$ be a simple graph. Suppose that each edge $e$ in $G$ is assigned a real number $f(e)$. For each vertex $v$, let $S(v)$ denote the sum of numbers assigned to the edges incident to $v$, that is, $$S(v)=\sum_{x\in N(v)}f(xv),$$where $N(v)$ is the set of neighbors of $v$. We say that $f$ is a \emph{cool decoration} of the edges of $G$ if $S(u)\neq S(v)$ for every pair of adjacent vertices in $G$. 

The following problem was posed in \cite{KaronskiLT JCTB} (see also \cite{Grytczuk DM}).

\begin{conjecture}[The 1-2-3 Conjecture]\label{Conjecture 1-2-3}
	Every connected graph with at least two edges has a cool edge decoration from the set $\{1,2,3\}$.
\end{conjecture}

The conjecture is optimal as the triangle $K_3$, for instance, cannot be decorated from the set $\{1,2\}$. It is known to be true for some classes of graphs (cliques, $3$-colorable graphs, etc., see \cite{KaronskiLT JCTB}). Currently best general result, due to Kalkowski, Karo\'{n}ski, and Pfender \cite{KalkowskiKP JCTB}, asserts that it holds for the set $\{1,2,3,4,5\}$. It is also known \cite{AddarioDR} that for random graphs the set $\{1,2\}$ is sufficient for a cool decoration (almost surely). Curiously, deciding if the set $\{1,2\}$ is actually sufficient for a given graph is computationally hard, as demonstrated by Dudek and Wajc in \cite{DudekWajc}. Recently Przyby\l o proved \cite{Przybylo} that regular graphs can be decorated from the set $\{1,2,3,4\}$, and also from the set $\{1,2,3\}$ if the vertex degree is at least $10^8$.

\subsection{Vertex decorations} Suppose that each vertex $v$ in a graph $G$ is assigned a real number $f(v)$. Let $S(v)$ be the sum of numbers assigned to the neighbors of $v$:  $$S(v)=\sum_{x\in N(v)}f(x).$$ We say that $f$ is a \emph{cool decoration} of the vertices of $G$ if $S(u)\neq S(v)$, for every pair of adjacent vertices in $G$.

Let $\chi(G)$ denote the chromatic number of a graph $G$. The following conjecture was stated in \cite{CzerwinskiGZ}.

\begin{conjecture}\label{Conjecture Additive}
	Every graph $G$ has a cool decoration of vertices from the set $\{1,2,\dots ,\chi(G)\}$.
\end{conjecture}

The conjecture is tight since every clique $K_n$ demands $n$ distinct numbers for a cool vertex decoration. It is not even known if the conjecture holds for bipartite graphs, even if the set of decorations is arbitrarily large.

Some results were established for planar graphs. For instance, planar bipartite graphs can be decorated from set the $\{1,2,3\}$ (see \cite{CzerwinskiGZ}). We will give a simple algebraic proof of a more general statement in section 2.5. Currently best result for general planar graphs \cite{BartnickiBCGMZ} uses the set $\{1,2,\dots,468\}$. Both results were obtained by using the algebraic method that will be described in section 2.5.

\subsection{Total decorations} Suppose now that each vertex $v$ and each edge $e$ of a graph $G$ are assigned real numbers $f(v)$ and $f(e)$, respectively. Let $S(v)$ be the sum of numbers assigned to the edges incident to $v$ plus the number $f(v)$: $$S(v)=f(v)+\sum_{x\in N(v)}f(vx).$$ We say that $f$ is a \emph{total cool decoration} of $G$ if $S(u)\neq S(v)$, for every pair of adjacent vertices in $G$.

This variant was introduced by Przyby\l o and Wo\'{z}niak in \cite{PrzybyloWozniak DMTCS}, where they stated the following conjecture.

\begin{conjecture}[The 1-2 Conjecture]\label{Conjecture 1-2}
	Every connected graph has a total cool decoration from the set $\{1,2\}$.
\end{conjecture}

We give a simple proof of a slightly weaker result of Kalkowski \cite{Kalkowski}. Actually, a similar idea led to the currently best bound in the original 1-2-3 Conjecture (see \cite{KalkowskiKP JCTB}).

\begin{theorem}[Kalkowski \cite{Kalkowski}]\label{Theorem Kalkos}
	Every graph has a total cool decoration with vertices decorated by the set $\{1,2\}$ and edges decorated by the set $\{1,2,3\}$.
\end{theorem}

\begin{proof} Let $v_1,v_2,\dots,v_n$ be any ordering of the vertices of $G$. Initially each vertex is decorated with number $1$, and each edge is decorated with number $2$. One may imagine that there is a chip lying on each vertex, while two chips are lying on each edge. We are going to refine this decoration so as to get a cool one by performing the following greedy procedure.
	
	To explain what we do in the $i$th step, denote by $x_1,x_2,\dots,x_k$ all backward neighbors of $v_i$, and let $e_j=v_ix_j$, with $j=1,2,\dots, k$, denote the corresponding backward edges. For each edge $e_j$ we have two possibilities: (1) if there is only one chip on $x_j$, then we may move one chip from $e_j$ to $x_j$ or do nothing, (2) if there are two chips on $x_j$ we may move one chip from $x_j$ to $e_j$ or do nothing. Notice that none of the sums $S(x_j)$ may change as a result of such action. Also, any action on each edge may change the total sum for $v_i$ just by one. Hence there are $k+1$ possible values for $S(v_i)$. So, at least one combination of chips gives a sum which is different from each of $S(x_j)$. We fix this combination and go to the next step. The proof is complete.
\end{proof}

\subsection{Decorations from lists} One may consider other sets of numbers for decorations of graphs. Recently, Vu\v{c}kovi\'{c} \cite{Vuckovic} proved the multi-set version of the 1-2-3 Conjecture, which implies that every graph has a cool edge decoration from any set of three real numbers, provided that this set is independent over rationals. It is not known if the 1-2-3 Conjecture is true for any $3$-element subset of the integers. However, by the the above result of Vu\v{c}kovi\'{c} \cite{Vuckovic} it follows that for every graph $G$ there is such set (depending on $G$). One may take, for instance, any set of the form $\{1,q,q^2\}$, where $q>\Delta(G)$ and $\Delta(G)$ is the maximum degree of $G$.

Suppose now that each edge $e$ of a graph $G$ is assigned its own list of allowable numbers $L(e)$. As in the list version of traditional graph coloring problem, we assume that decoration of $e$ must be taken from $L(e)$.

The following statement is a strengthening of the 1-2-3 Conjecture proposed in \cite{BartnickiGN JGT}. 

\begin{conjecture}[The List 1-2-3 Conjecture]\label{Conjecture 1-2-3 Lists}
	Every connected graph with at least two edges has a cool edge decoration from arbitrary lists of size three.
\end{conjecture}

The problem is wide open. It is not known if there is any finite bound for list sizes guaranteeing a cool edge decoration. Some results were obtained in \cite{BartnickiGN JGT} by using the algebraic method which is described in the next subsection.

\subsection{Combinatorial Nullstellensatz}
Let $G$ be a simple graph with $m$ edges. Assign a variable $x_e$ to each edge $e$ of a graph $G$, and consider a polynomial
\begin{equation}\label{Polynomial 1-2-3}
P=\prod_{uv\in E(G)}(S(u)-S(v)),
\end{equation}
where $S(v)$ is the sum of variables assigned to the edges incident to the vertex $v$. We consider $P$ as a polynomial over the field $\mathbb R$ of real numbers. Clearly, any substitution for variables $x_e$ from lists $L(e)\subseteq \mathbb R$ giving a non-zero value of $P$ is a cool decoration of $G$. Thus, Conjecture \ref{Conjecture 1-2-3 Lists} will follow if we could prove that $P$ does not vanish over any grid $A_1\times A_2\times \dots \times A_m$, with $|A_i|=3$, $A_i\subseteq \mathbb R$.

If $P$ is a one-variable polynomial of degree $k$ (over any field $\mathbb F$) and $A$ is any set of at least $k+1$ elements from $\mathbb F$, then $P$ cannot vanish at all members of $A$. The following theorem of Alon \cite{Alon CN} is an elegant generalization of this simple fact for multivariable polynomials.

\begin{theorem}[Combinatorial Nullstellensatz \cite{Alon CN}] Let $P$ be a polynomial in $\mathbb F[x_1,x_2,\dots,x_m]$ over any field $\mathbb F$. Suppose that there is a non-vanishing monomial $x_1^{k_1} x_2^{k_2}\cdots x_m^{k_m}$ in $P$ whose degree is equal to the degree of $P$. Then, for arbitrary sets $A_i\subseteq \mathbb F$, with $|A_i|=k_i+1$, there is a choice of elements $a_i\in A_i$ such that $P(a_1,a_2,\dots,a_m)\neq 0$.
\end{theorem}

Notice that the polynomial $P$ defined in (\ref{Polynomial 1-2-3}) is \emph{uniform} (all monomials have the same degree), and moreover, its degree is equal to the number of variables in $P$. So, there is a chance for a non-vanishing monomial with all exponents equal to one, which would imply a cool decoration from lists of size two. Using this method and some ideas from \cite{BartnickiGN JGT}, Wong and Zhu \cite{WongZhu C} proved the following list version of Theorem \ref{Theorem Kalkos}.

\begin{theorem}[Wong and Zhu \cite{WongZhu C}]\label{Theorem Kalkos CN Zhu}
	Every graph has a total cool decoration form any lists of size two assigned to the vertices and any lists of size three assigned to the edges.
\end{theorem}

This is very close to the natural list analogue of the 1-2 Conjecture, stated by Przyby\l o and Wo\'{z}niak in \cite{PrzybyloWozniak ElJC}, and independently by Wong and Zhu in \cite{WongZhu JGT}.

\begin{conjecture}[The List 1-2 Conjecture]\label{Conjecture 1-2-Lists}
	Every graph has a total cool decoration from any lists of size two assigned to the vertices and edges.
\end{conjecture}

To illustrate the method based on Combinatorial Nullstellensatz, we give a simple proof of the result from \cite{CzerwinskiGZ} on vertex decorations of planar bipartite graphs.

\begin{theorem}[Czerwi\'{n}ski, Grytczuk, and \.{Z}elazny \cite{CzerwinskiGZ}]\label{Theorem Vertex Planar Bipartite}
	Every planar bipartite graph has a cool vertex decoration from any lists of size three assigned to the vertices.
\end{theorem}

\begin{proof}
	Let $G$ be a planar bipartite graph with bipartition classes $X$ and $Y$. Assign to each vertex $u$ in $X$ a variable $x_u$ and to each vertex $v$ in $Y$ a variable $y_v$. Let $S(u)$ denote the sum of variables assigned to the neighbors of $u$. Consider the polynomial
	
	\begin{equation}
	P=\prod_{uv\in E(G),u\in X, v\in Y}(S(u)-S(v)).
	\end{equation}
	Notice that all variables for $X$ appear with minus sign in the sum $S(u)-S(v)$, while variables for $Y$ appear with plus sign. This implies that all monomials in $P$ with the same sequence of exponents must have the same sign. Hence, none of monomials formed by choosing one variable from each factor $(S(u)-S(v))$ will eventually vanish in $P$ (as $P$ is a polynomial over the filed of real numbers).
	
	To complete the proof it suffices to demonstrate that there is a choice of variables giving a monomial with at most quadratic exponents. To see this recall that a planar bipartite graph on $n$ vertices can have at most $2n-4$ edges, and therefore it can be oriented so that each vertex has at most two incoming edges. Thus, we may form a desired monomial by choosing from each factor $(S(u)-S(v))$ either $x_u$ or $y_v$ according to the orientation of the edge $uv$. By Combinatorial Nullstellensatz, this completes the proof.
\end{proof}

The same argument gives the well-known result of Alon and Tarsi \cite{AlonTarsi} on $3$-choosability of planar bipartite graphs. 

\subsection{Ironic decorations}
This variation on the 1-2-3 Conjecture, proposed by \.{Z}elazny, is slightly different from previous as it uses multiplication of numbers instead of addition.

Suppose that each vertex $v$ of a graph $G$ is assigned a real number $f(v)$. Let $M(v)=f(v)d_v$ be the product of the assigned number by the degree $d_v$ of the vertex $v$. We say that $f$ is an \emph{ironic decoration} of $G$ if $M(u)\neq M(v)$ for every pair of adjacent vertices in $G$.

The following conjecture was stated in \cite{BosekDGSSZ DM}.

\begin{conjecture}\label{Conjecture Iron}
	Every graph $G$ has an ironic decoration by the set $\{1,2,\dots, \chi(G)\}$.
\end{conjecture}

Notice that the conjecture is trivially true (and optimal) for regular graphs. Also it may be stated as a special case of a restricted list coloring problem. Suppose that each vertex $v$ is assigned a list $L(v)=\{d_v,2d_v,\dots,kd_v\}$, where $d_v$ is the degree of $v$. In other words, the list has a form of a \emph{homogenous arithmetic progression}. Then an ironic decoration of $G$ can be extracted from any proper coloring of $G$ from these lists by striking out factors $d_v$.

This observation leads to a more general problem. For brevity, any set of the form $\{r,2r,\dots,kr\}$ will be called a \emph{cascade} of length $k$ (\emph{rooted} at $r$).

\begin{conjecture}[The List Cascade Coloring]\label{Conjecture List Cascade}
	Every graph $G$ is colorable from arbitrary cascades of length $\chi(G)$.
\end{conjecture}

One natural approach to this problem leads unexpectedly to some deep number theoretic questions. We present this connection in the next section.

\section{Rainbow Cascades and Arithmetic Graphs}

\subsection{Rainbow Cascades}
Suppose that we are given a graph $G$ with some lists $L(v)\subseteq \mathbb N$ assigned to the vertices, each list of size $k$. Suppose further that there is a $k$-coloring of $\mathbb N$ such that every list $L(v)$ is \emph{rainbow} (no two elements in the list have the same color). Then, if a graph $G$ is $k$-colorable, it can also be colored from lists $L(v)$. Indeed, if a vertex $v$ is colored red in a proper coloring of $G$, then just take a red element from the list $L(v)$ and assign it to $v$ as a color from chosen form its list.

This observation leads to the following conjecture.

\begin{conjecture}[The Rainbow Cascades Conjecture]\label{Conjecture Rainbow Cascade}
	For every $k\in \mathbb N$ there is a $k$-coloring of $\mathbb N$ such that every cascade of length $k$ is rainbow.
\end{conjecture}

This conjecture was posed independently (as a question) by Pach and P\'{a}lv\"{o}lgyi (see \cite{CaicedoCP}). Below we give a simple proof for cascades of length $p-1$, where $p$ is a prime number.

\begin{theorem}\label{Theorem Cascades p-1}
	Let $p$ be a prime number. Then there exists a $(p-1)$-coloring of $\mathbb N$ such that every cascade of length $p-1$ is rainbow.
\end{theorem}
\begin{proof} We define a desired coloring as follows. Write a natural number $n$ as $n=p^sm$, where $m$ is not divisible by $p$. Let $r(m)$ be the residue of $m$ modulo $p$. We assign $r(m)$ as a color of the number $n$. Since residue zero is excluded, there are $p-1$ different colors.
	
	It is not hard to see that no two elements of the same cascade may have the same color. Indeed, let $an$ and $bn$ be any two distinct elements of cascade $\{n,2n,\dots, (p-1)n\}$ rooted at $n$. Since $a$ and $b$ are not divisible by $p$, we have $an=p^sma$ and $bn=p^smb$. Consequently, the color of $an$ is $r(ma)$ and the color of $bn$ is $r(mb)$. If these two colors are equal, then, by multiplication properties of residues modulo $p$, also $r(a)=r(b)$, which means that $a=b$. Hence, $an=bn$ and the proof is complete.
\end{proof}

Besides the case $k=p-1$ the Rainbow Cascades Conjecture is known to hold for those values of $k$ for which one may define a group multiplication on the set $\{1,2,\dots, k\}$ compatible with ordinary multiplication inside this set. In particular, it holds for $k=(p-1)/2$, $k=p^2-p$, where $p$ is a prime, and all $k$ up to $194$ (see \cite{CaicedoCP}).

\subsection{Graham's gcd-Problem}
In \cite{Graham AMM} Graham posed the following problem. Let $a_1,a_2,\dots,a_n$ be any distinct positive integers. Prove that some pair $a_i,a_j$ satisfies:

\begin{equation}
\frac{a_i}{\gcd (a_i,a_j)}\geqslant n.
\end{equation}

The problem was solved for sufficiently large $n$ by Szegedy \cite{Szegedy} and independently by Zaharescu \cite{Zaharescu}. Then Balasubramanian and Soundararajan \cite{BalasubramanianS} gave a complete solution by using methods of analytic number theory.

\subsection{Arithmetic Graphs}
The problem of ironic decorations unexpectedly appeared to be related to Graham's gcd-Problem. In an attempt to solve Conjecture \ref{Conjecture Iron}, Bosek defined auxiliary graphs reflecting in some sense the arithmetic proximity of numbers. To present his idea let us define \emph{arithmetic proximity} between two integers $a$ and $b$ as $\max \{a/d,b/d\}$, where $d=\gcd (a,b)$. So, two numbers are arithmetically close if both results of division by their greatest common divisor are relatively small.

Let $k$ be a fixed positive integer. Define a graph on the set $\mathbb N$ by joining $a$ to $b$ if and only if their arithmetic proximity is at most $k$. We will denote these graphs as $B_k$ and call them \emph{arithmetic graphs}.

Motivated by Conjecture \ref{Conjecture Iron}, Bosek posed the following conjecture (see \cite{BosekDGSSZ DM}).

\begin{conjecture}\label{Conjecture Bosek Graphs}
	Every arithmetic graph $B_k$ satisfies $\chi (B_k)=k$.
\end{conjecture}

It is not hard to see that this conjecture is stronger than the statement of Graham's problem. Indeed, recall that the \emph{clique number} $\omega(G)$ of graph $G$ is the size of a largest clique in $G$. Notice that Graham's conjecture is equivalent to the statement that $\omega(B_k)=k$. Indeed, taking $n=k+1$ in Graham's problem, we see that among any $k+1$ vertices there must be a pair which is not joined by an edge. Since every graph satisfies $\omega(G)\leqslant \chi(G)$, Bosek's conjecture is an extension of Graham's conjecture.

It is easily seen that Bosek's conjecture is equivalent to the Rainbow Cascades Conjecture. Indeed, $ab$ is an edge in $B_k$ if and only if the numbers $a$ and $b$ belong to the same cascade of length $k$ (rooted at $d=\gcd (a,b)$). Hence, by Theorem \ref{Theorem Cascades p-1} we know that $\chi (B_{p-1})=p-1$ for every prime number $p$. A famous Bertrand's Postulate (asserting that there is a prime between $n$ and $2n$) implies that $\chi(B_k)\leqslant 2k$. Using more exact results on primes in short intervals one may deduce that $\chi(B_k)=(1+o(1))k$ (see \cite{BosekDGSSZ DM}).

\section{Variations on the Erd\H {o}s Discrepancy Problem}

\subsection{Balanced Cascades}

Consider a coloring $f$ of $\mathbb N$ by two colors $\{-1,+1\}$. For a finite subset $A\subseteq \mathbb N$, the number $b_f(A)=|\sum_{x\in A}f(x)|$ is called the \emph{balance} of a coloring $f$ on $A$. For $A=\{1,2,\dots, n\}$ we write $b_f(A)=b_f(n)$. A set $A$ is called \emph{$C$-balanced} if $b_f(A)\in \{0,1,\dots, C\}$ for some integer $C>0$. If $b_f(A)\in \{0,1\}$, then $A$ is said to be \emph{balanced}.

As a weakening of the Rainbow Cascades Conjecture, Bosek asked the following question: \emph{Is there a constant $C$ such that for every $k$ there is a $2$-coloring of $\mathbb N$ in which every cascade of length $k$ is $C$-balanced?} If the Rainbow Cascades Conjecture is true, then we may arbitrarily split the set of colors into two almost equal parts and each rainbow cascade becomes balanced. So, by Theorem \ref{Theorem Cascades p-1} we know that this holds for $k=p-1$, where $p$ is a prime number. 

The following theorem extends this result for arbitrary $k$, answering Bosek's question in the affirmative (see \cite{BosekDGKLZ}).

\begin{theorem}\label{Theorem Balanced Cascades}
	For every $k$ there is a $2$-coloring of $\mathbb N$ such that every cascade of length $k$ is balanced.
\end{theorem}

We will sketch the proof of this result below.

\subsection{Multiplicative coloring} Notice that the coloring $r$ from the proof of Theorem \ref{Theorem Cascades p-1} is \emph{multiplicative}, which means that for every pair of integers $a$ and $b$, we have $r(ab)=r(a)r(b)$ modulo $p$. It follows that we may construct a balanced $2$-coloring with the same property by taking a suitable partition of $\{1,2,\dots, p-1\}$ into "positive" and "negative" elements. This suggests the following approach in general case.

Let $f:\mathbb N\rightarrow \{-1,+1\}$ be a $2$-coloring of positive integers. We say that $f$ is \emph{multiplicative} if $f(ab)=f(a)f(b)$ for every pair $a,b\in \mathbb N$. Notice that, by the uniqueness of prime factorization, any such coloring is determined by fixing colors of prime numbers. Moreover, $f(1)=+1$ must be satisfied. It is also worth noting that each cascade $\{d,2d,3d,\dots\}$ repeats either $f$ or $-f$, that is, $$(f(d),f(2d),f(3d),\dots)=(f(1),f(2), f(3),\dots)$$ or $$(f(d),f(2d),f(3d),\dots)=(-f(1),-f(2),-f(3),\dots),$$ according to whether $f(d)=+1$ or $-1$, respectively.

The most natural multiplicative coloring is obtained by putting $-1$ for every prime number. This is the well known \emph{Liouville function}, denoted as $\lambda(n)$. So, $\lambda (n)$ is positive or negative according to whether $n$ is a product of even or odd number of primes (counting with multiplicity), respectively. For instance, $\lambda(12)=\lambda (2\cdot2\cdot3)=-1$, while $\lambda (40)=\lambda(2\cdot2\cdot2\cdot5)=+1$.

Estimating the balance $b_{\lambda}(n)$ of this coloring is an important and difficult issue. For instance, proving that $b_{\lambda}(n)=o(n)$ is already equivalent to the Prime Number Theorem, while $b_{\lambda}(n)=O(n^{1/2+\epsilon})$ for every $\epsilon>0$ is equivalent to the Riemann Hypothesis (see \cite{BorweinCRW}).

\subsection{Golden Seeds}

Let $k$ be a fixed positive integer. To construct a $2$-coloring of $\mathbb N$ which is balanced on cascades of length $k$ it is enough to find a binary string $S=s_1s_2\dots s_k$ over $\{-1,+1\}$ with the following properties:
\begin{itemize}
	\item[(i)] $s_{ij}=s_is_j$, whenever $ij\leqslant k$,
	\item[(ii)] $\sum_{i=1}^{k}{s_i}\in \{-1,0,1\}$.
\end{itemize}
We will call such strings \emph{golden seeds}. For each $k\geqslant 2$ there are $2^{\pi(k)}$ strings satisfying condition (i), where $\pi(k)$ counts the number of primes in the set $\{1,2,\dots, k\}$. We are going to prove that at least one of them satisfies also condition (ii). For instance, for $k=6$, there are three golden seeds out of eight strings satisfying condition (i):
$$	\begin{array}{|c|c|c|c|c|c|}
\hline
1 & 2 & 3 & 4 & 5 & 6 \\
\hline
+  & - & - & + & - & +\\
+  & - & + & + & - & -\\
+ & + & - & + & -& -\\
\hline
\end{array}
$$
This will be enough to get Theorem \ref{Theorem Balanced Cascades}. Indeed, suppose that $S$ is a golden seed of length $k$. One may extended $S$ to a multiplicative 2-coloring $f$ of $\mathbb N$ by taking $f(i)=s_i$ for $i=1,2,\dots,k$, and anything for primes greater than $k$ (other values of $f$ are determined by multiplicativity). Now, if $A=\{d,2d,\dots, kd\}$ is any cascade of length $k$, then
$$f(d)+f(2d)+\dots+f(kd)=f(d)\sum_{i=1}^{k}{f(i)}=f(d)\sum_{i=1}^{k}s_i\in \{-1,0,1\},$$by condition (ii).

Theorem \ref{Theorem Balanced Cascades} follows from the following result (see \cite{BosekDGKLZ}).

\begin{theorem}\label{Theorem Golden Seeds}For every $k\geqslant 2$ there exists a golden seed of length $k$.
\end{theorem}
\begin{proof}[Proof (sketch)]
	A basic idea of the proof is simple. We start with a multiplicative $2$-coloring $g$ of positive integers specified by taking $g(p)=\pm 1$ in accordance to whether a prime $p$ is congruent to $+1$ or $-1$ modulo $3$, with $g(3)=+1$. It can be proved that $\sum_{i=1}^{k}g(i)$ is exactly equal to the number of $1$'s in the ternary expansion of $k$ (see \cite{BorweinCC}). In particular, this sum is never negative and bounded from above by $\log_3k+1$. So, to get a golden seed $S$ of length $k$ it suffices to change the sign $+1$ into $-1$ of at most $\log_3k+1$ primes of the form $3t+1$ lying in the interval $[k/2,k]$. This operation will not affect multiplicativity of $S$. That there exists sufficient number of primes of that form in this interval follows from the celebrated Dirichlet's theorem on primes in arithmetic progressions (see \cite{Apostol}). This gives the result for sufficiently large $k$. Complete proof demands more delicate tricks together with some computational experiments (see \cite{BosekDGKLZ}).
\end{proof}

\subsection{Rejmer's Algorithm}To prove Theorem \ref{Theorem Golden Seeds} we considered firstly a different approach proposed by Rejmer. It is a simple algorithm producing golden seeds in a greedy way.

Let $k$ be a fixed positive integer. Our aim is to construct a golden seed $S$ of length $k$. We start with putting $s_1=+1$. In each consecutive step we add new sign trying to preserve both properties, balance and multiplicativity. So, in the second step we put $s_2=-1$.

Suppose that after $j-1$ steps, $j>2$, we obtained a golden seed $s_1s_2\dots s_{j-1}$. We distinguish two cases.

\begin{itemize}
	\item [1.] If $j-1$ is even, then $\sum_{i=1}^{j-1}s_i=0$. Thus no choice for $s_j$ may destroy balance. If $j$ is composite, then $s_j$ is determined by multiplicativity. If $j$ is prime, then we put $s_j=-1$.
	\item [2.] If $j-1$ is odd, then $j$ is even, so $s_j$ is determined by multiplicativity. Since $\sum_{i=1}^{j-1}s_i=\pm 1$, we may have either $\sum_{i=1}^{j}s_i=0$ or $\sum_{i=1}^{j}s_i=\pm 2$. In the former case we are done. In the later case we look for the largest prime $p>j/2$ for which $s_p$ has wrong sign and switch it. This makes the string $s_1s_2\dots s_j$ balanced.
\end{itemize}

We proceed similarly in next steps until obtaining a golden seed of a desired length. Notice that by Bertrand's Postulate, there is always a prime between $j/2$ and $j$. It is ever not clear that there will always be a prime whose sign-switching would improve balance. For instance, in the $16$th step of the algorithm we get the following string:
$$	\begin{array}{|c|c|c|c|c|c|c|c|c|c|c|c|c|c|c|c|}
\hline
1&2&3&4&5&6&7&8&9&10&11&12&13&14&15&16 \\
\hline
+&-&-&+&-&+&-&-&+&+&+&-&-&+&+&+\\
\hline
\end{array}
$$
which has $9$ pluses and $7$ minuses. To fix this imbalance, we go back to the first prime to the left, which is $13$. However, the sign of $13$ is $-$, so switching it would only increase imbalance. Fortunately, the next prime is $11$ with $+$, so we may switch it to get a balanced multiplicative string.

We do not known if Rejmer's algorithm runs ad infinitum.

\begin{conjecture}\label{Conjecture Rejmer}
	Rejmer's algorithm never stops.
\end{conjecture}

Rejmer made some computational experiments with his algorithm. In particular, he run it up to $10^9$ steps producing in this way a golden seed of that length (and all smaller lengths on the way). Notice that the first half terms of this seed will not be changed in the future. Thus, assuming validity of Conjecture \ref{Conjecture Rejmer}, the algorithm defines an intriguing recursive binary sequence $R(n)$ over $\{-1,+1\}$. Up to $n=40$ Rejmer's sequence coincides with the Liouville function $\lambda(n)$, but $R(41)=+1$. The same happens for many other primes, in particular $R(97)=R(101)=+1$. One may suspect that there will be infinitely many primes $p$ with $R(p)=+1$, as well as with $R(p)=-1$.

\subsection{The Erd\H {o}s Discrepancy Problem}

In 1932 Erd\H {o}s posed an intriguing problem (see \cite{Erdos}): \emph{Is there a constant $C$ and a $2$-coloring of $\mathbb N$ such that every finite cascade is $C$-balanced?} This is much stronger property than in Bosek's question as it asks for one coloring that will be good for \emph{all} finite cascades. This seems unbelievable and actually Erd\H {o}s conjectured that the answer is negative. It took long time and many efforts until Tao \cite{Tao} finally proved this conjecture (see \cite{Soundararajan}).

\begin{theorem}[Tao \cite{Tao}]\label{Theorem Tao}
	In every $2$-coloring of $\mathbb N$ there are cascades of arbitrarily large balance.
\end{theorem}

A subset $B$ of positive integers is called \emph{balanceable} if there is a constant $C$ and $2$-coloring of $\mathbb N$ such that every cascade whose length is in $B$ is $C$-balanced. So, by Theorem \ref{Theorem Balanced Cascades} we know that every singleton is balanceable, while by Tao's result the whole set $\mathbb N$ is not.

How big can a balanceable set be? It is not hard to see that it can be infinite. For instance, let $B$ be the set of positive integers whose ternary expansion does not contain $1$'s. If $g$ is a multiplicative function from the proof of Theorem \ref{Theorem Golden Seeds}, then for every $k\in B$ we know that $\sum_{i=1}^{k}g(i)=0$. Therefore a $2$-coloring defined by $g$ is balanced on each cascade whose length is in $B$. The set $B$ is however of density zero in $\mathbb N$. Is there a \emph{dense} balanceable subset $B\subseteq \mathbb N$?

\begin{conjecture}\label{Conjecture Balanceable}
	There exists a balanceable set of positive density.
\end{conjecture}

Analogous problem can be stated for multiplicative colorings. By Theorem \ref{Theorem Tao} we know that there is no multiplicative $2$-coloring $f$ with bounded balance $b_f(n)$. But maybe for some constant $C$ there is a multiplicative $2$-coloring $f$ for which the set $\{n\in \mathbb N:b_f(n)\leqslant C\}$ has positive density.

\subsection{Divine version of the Erd\H {o}s Discrepancy Problem} The following problem was posed by Bosek. Let $f$ be a coloring of $\mathbb N$ by two colors $\{-1,+1\}$. A subset $A=\{a_1,a_2,\dots,a_n\}$ of $\mathbb N$, with $a_1<a_2<\dots<a_n$, is \emph{divinely} colored if at least half of elements among $a_2,a_3,\dots,a_n$ have different color than $a_1$. For non-singletons this means that if $f(a_1)=+1$, then $\sum_{x\in A}f(x)\leqslant 1$, while $\sum_{x\in A}f(x)\geqslant -1$ when $f(a_1)=-1$. In particular, a balanced set is divinely colored. This notion was inspired by \emph{majority coloring} of digraphs (see \cite{KreutzerOSZW}).

\begin{conjecture}\label{Conjecture Divine}
	There exists a $2$-coloring of $\mathbb N$ in which every cascade is divinely colored.
\end{conjecture}

Let us see what this property means for multiplicative colorings. Since $f(1)=+1$, we must have $\sum_{i=1}^{n}f(i)\leqslant 1$ for all $n\geqslant 1$. This obviously guarantee that every cascade is divinely colored as the set $d,2d,3d,\dots$ carries the same coloring $f$ or its negative $-f$, according to the sign of $d$.

Is there a multiplicative $2$-coloring of $\mathbb N$ whose partial sums are always bounded by one? A natural candidate is the Liouville function $\lambda(n)$, as it is defined by $\lambda(p)=-1$ for each prime $p$. Actually in 1919 P\'{o}lya \cite{Polya} conjectured that $\sum_{i=1}^{n}\lambda(i)\leqslant 0$ for all $n\geqslant  2$, which if true, would imply the Riemann Hypothesis. Unfortunately the statement is not true, but the smallest counter-example is $n=906150257$ (see \cite{BorweinCRW}).

As in the original Erd\H {o}s Discrepancy Problem, one may consider a relaxed version of divine coloring with some constant error $C$. This leads to the following problem.

\begin{conjecture}\label{Conjecture Bounded Multiplicative}
	There exists a constant $C$ and a multiplicative $2$-coloring $f$ of $\mathbb N$ by colors $\{-1,+1\}$ such that $\sum_{i=1}^{n}f(i)\leqslant C$ for all $n\geqslant 1$.
\end{conjecture}

Some massive computer experiments were made in \cite{BosekDGKLZ}. Curiously, "more negative" than Liouville's function are colorings obtained by switching sign of just one small prime number. However, it can be proved that switching signs of a finite subset of primes gives a function whose negativity would still imply the Riemann Hypothesis, in much the same way as it is for the Liouville function.

\section{Back to the 1-2-3 Conjecture}
We conclude the paper with one more problem related to the 1-2-3 Conjecture. Let $\mathcal F$ be a \emph{hypergraph}, that is, a family of subsets of some finite set $V$. Suppose that each element $v\in V$ is assigned a real number $f(v)$. For every set $A\in \mathcal F$, let $S(A)$ denote the sum of numbers assigned to the vertices of $A$:$$S(A)=\sum_{x\in A}f(x).$$We say that $f$ is a \emph{cool decoration} of $\mathcal F$ if $S(A)\neq S(B)$ for every pair $A,B\in \mathcal F$ whose intersection is non-empty.

The following problem is a natural generalization of the 1-2-3 Conjecture for hypergraphs (see \cite{BennetDFH ElJC}, \cite{KalkowskiKP JGT} for other variants).

\begin{conjecture}\label{Conjecture 1-2-3 Hypergraphs}
	For every $\Delta$ there is a constant $C=C(\Delta)$ such that every hypergraph $\mathcal F$ of maximum degree $\Delta$ has a cool decoration from the set $\{1,2,\dots,C\}$.
\end{conjecture}

The original 1-2-3 Conjecture is obtained by taking as $\mathcal F$ a dual hypergraph to a given graph $G$. The maximum degree of such hypergraph is $\Delta(\mathcal F)=2$, as every edge in a simple graph has two ends. As observed by Przyby\l o, Conjecture \ref{Conjecture 1-2-3 Hypergraphs} is true in general for $\Delta=2$ (with constant $C=5$, as yet). It is wide open even for $\Delta=3$.

\end{document}